\documentclass[english, 11pt]{amsart}

\usepackage{tikz}
\usepackage{aliascnt}
\usepackage{amsfonts}
\usepackage{mathrsfs}
\usepackage{amsthm}
\usepackage{amsmath,amssymb}
\usepackage[all,poly,knot]{xy}
\usepackage{amsfonts}
\usepackage{color}
\usepackage{hyperref}
\usepackage{comment}  
\usepackage{csquotes}

\theoremstyle{plain}
\newtheorem{thmx}{Theorem} 

\newtheorem{thm}{Theorem}[section]  

\newtheorem{cor}[thm]{{Corollary}} 
\newtheorem{corx}[thmx]{Corollary}
\newtheorem{lem}[thm]{{Lemma}}

\newtheorem{prop}[thm]{Proposition}

\newtheorem{defi}[thm]{Definition}

\theoremstyle{remark}
\newtheorem{rmk}[thm]{Remark}

\numberwithin{equation}{section}

\def\Spec{\mathrm{Spec}}

\def\log{\mathrm{log}\,}

\def\Sym{\mathrm{Sym}}

\usepackage[top=1.5in, bottom=1.5in, left=1in, right=1in]{geometry}

 \usepackage[hyperpageref]{backref}
 
\usepackage{xcolor}
\hypersetup{
	colorlinks,
	linkcolor={red!50!black},
	citecolor={blue!50!black},
	urlcolor={blue!80!black}
}
\begin{document} 
\title[Holomorphic curves in Base Spaces]{Holomorphic curves in Base Spaces of Families of Polarized Manifolds} 
	\author{Steven Lu}
		 \address{D\'epartement de math\'ematiques
		 	Universit\'edu Qu\'ebec \`a Montr\'eal
		 	Case postale 8888, succursale centre-ville
		 	Montréal (Québec) H3C 3P8}
	\email{lu.steven@uqam.ca}
	\author{Ruiran Sun}
		 \address{Institut fur Mathematik, Universit\"at Mainz, Mainz, 55099, Germany}
	\email{ruirasun@uni-mainz.de}
        \author{Kang Zuo}
        	 \address{Institut fur Mathematik, Universit\"at Mainz, Mainz, 55099, Germany}
        \email{zuok@uni-mainz.de}

\begin{abstract}
For a smooth family $V \to U$ of polarized manifolds with semi-ample canonical sheaves, we show the following result: any entire curve must be contained in the fibers of the classifying map from the base space $U$ to the moduli space. This settles the  Relative Isotriviality Conjecture, \cite[Conjecture 1.5]{DLSZ}.           
\end{abstract}


\maketitle

\section{Introduction}
In \cite{VZ-1} Viehweg and the third named author proved the Brody hyperbolicity of the moduli stack of canonically polarized complex manifolds, that is, for any such family $V \to U$ with quasi-finite classifying map, there is no nonconstant entire curves $\mathbb{C} \to U$.\\
In this notes, we study the distribution of entire curves in the base manifold of family of polarized manifolds with \emph{non-maximal} variation. Our main result is
\begin{thmx}[= Theorem~\ref{main-thm}]\label{thm-A}
Let $(f:\,V \to U,\mathcal{L})\in \mathscr{M}_h(U)$ be a smooth family of polarized manifolds with semi-ample canonical sheaves and fixed Hilbert polynomial $h$.  Let $\varphi:\, U \to \mathcal{M}_h$ be the induced classifying map from the base $U$ of the family to the coarse moduli space $ \mathcal{M}_h$. Let $\gamma:\, \mathbb{C} \to U$ be an entire curve. Then the image curve $\gamma(\mathbb{C})$ is contained in a fiber of $\varphi$.  
\end{thmx}
The basics of the moduli functor $\mathscr{M}_h$ and the associated coarse moduli scheme $\mathcal{M}_h$ are recalled in section~\ref{compatibility}. As a direct corollary of Theorem~\ref{thm-A}, we obtain
\begin{corx}\label{cor-B}
If the base space $U$ contains a Zariski dense entire curve, then the family $f:\, V \to U$ is isotrivial.  
\end{corx}
Corollary~\ref{cor-B} can be regarded as \emph{a hyperbolic version of Campana’s isotriviality conjecture} \cite[Conjecture~13.21]{Cam11}. In Campana's original conjecture the base space is assumed to be \emph{speical}. See the joint paper of the first named author with Winkelmann \cite{LW} about the relation of special varieties and the property of containing Zariski dense entire curves. This is the reason that we refer Theorem~\ref{thm-A} as a relative isotriviality theorem.\\[-2mm]

Theorem~\ref{thm-A} has been proved  for families of canonically polarized manifolds in \cite{Den19}. There, the Weil-Petersson metric was used to show the moduli stack is hyperbolic as an orbifold. However, there is  in general no Weil-Petersson metric for a family of polarized manifolds. Our approach is roughly to apply Viehweg-Zuo's original argument but ``on the image $\varphi(U)$'' to show that the composed entire curve $\varphi \circ \gamma$ is constant.\\[-3mm]

We briefly explain our strategy. The property of the moduli scheme and a result of Koll\'ar on the existence of fine moduli spaces enable us to construct a commutative diagram of surjective maps
\[
\xymatrix{
U' \ar[r] \ar[d] & U^{\#} \ar[d]\\
U \ar[r] & \varphi(U)
}
\] 
where $U'$ is finite over $U$ and $U^{\#}$ carries a family with \emph{maximal variation}. Furthermore, the base changes to $U'$ of the family over $U$ and that over $U^{\#}$ coincide.  Lemma~\ref{diagram}, the technical basis of our arguments, provides good birational models of these base spaces so that we can study their geometric relationships. With it,  we follow the line of reasoning in \cite{VZ-1} to construct complex Finsler pseudometrics ``on $U^{\#}$" with nice curvature properties. Regarding the finite cover $U' \to U$ as a correspondence from $U$ (to $U^{\#}$), we need to consider the ``ramified entire curve'' $\mathbb{C}'=$ normalization of $\mathbb{C}\times_U U'$, regarded as a holomorphic correspondance from $\mathbb{C}$ in the diagram 
\[
\xymatrix{
\mathbb{C}' \ar[r]^{\gamma'} \ar[d] & U' \ar[r]^{\eta} \ar[d] & U^{\#} \\
\mathbb{C} \ar[r]^{\gamma} & U  &\ \ \ \ .
}
\]\\[1mm]
The key in our metric argument is that the ``differential" of this holomorphic correspondance factors through $\mathbb{C}'$, resulting in our key technical Lemma~\ref{factor-through}.
The final step is to use the complex Finsler pseudometric on $Y^{\#}$ to show that the ramified entire curve $\eta \circ \gamma'$ is constant. Just like \cite{VZ-1}, we need a ramified version of Ahlfors-Schwarz lemma on $\mathbb{C}'$. This lemma is given in section~\ref{AS-lem}.\\

Our second result concerns the distribution of more general holomorphic curves in base manifolds of smooth polarized families. We first give a definition
\begin{defi}[Borel-Chern Hyperbolicity]
A complex variety $Y$ is said to be Borel-Chern hyperbolic if for any algebraic curve $C$, the graph $\Gamma_{\gamma}$ of every holomorphic map $\gamma:\, C \to Y$ is not Zariski densen in $C \times Y$.
\end{defi}
From the definition one notices that if $\mathrm{dim}\,Y=1$, it is equivalent to the Borel hyperbolicity of $Y$: for any algebraic curve $C$, every holomorphic map $\gamma:\, C \to Y$ is an algebraic morphism.\\
Now we can state our second main result:
\begin{thmx}\label{thm-C}
Let $(f:\,V \to U,\mathcal{L})\in \mathscr{M}_h(U)$ be a smooth family of polarized manifolds with semi-ample canonical sheaves and fixed Hilbert polynomial $h$. Suppose that the family is non-isotrivial, i.e. the induced classifying map $\varphi:\, U \to \mathcal{M}_h$ is non-constant. Then the base space $U$ is Borel-Chern hyperbolic.
\end{thmx}

\noindent{\bf Acknowledgment.} The second named author thanks Ya Deng for the explanation of orbifold structures and his proof in \cite{Den19}.
The second and third named authors gratefully acknowledge support from SFB/Transregio 45 while the first named author thanks NSERC and CIRGET for theirs.

\section{Ramified Pseudometric and the Ahlfors-Schwarz Lemma}\label{AS-lem}
Let $C, C'$ be Riemann surfaces. Let $\psi_C:\, C' \to C$ be a holomorphic {\em ramified covering map}, i.e. proper and finite onto its image when restricted to each component of the preimage of every compact subset of $C$. For any point $p \in C'$, we denote by $\mathrm{Ram}(\psi_C,p):= \deg_p\psi_C-1$ the ramification index of $\psi_C$ at the point $p$. The set R$_\psi:=\{p \in C'\,\,|\,\, \mathrm{Ram}(\psi_C,p) \geq 1\}$ is discrete.

\begin{defi}\label{def_ram_metric}
We call a complex 2-tensor $g$ a {\bf ramified pseudometric} on $C'$ with respect to $\psi:\,C' \to C$ if it is conformally equivalent to the pull back of a hermitian metric $h$ on $C$ with a conformal factor $k\ge 0$ that is $C^2$ away from its zero set, i.e. $g=k\psi^*h$ for a continuous function $k$ on $C'$ that is $C^2$ away from its zero set.
Let $(U',z)$ be a local chart of $C'$ centred at $p'\in C'$ and $(U,t)$ be a local chart of $C$ centred at $\psi(p')$. Then $g$  can be written in $V:=U'\cap\psi^{-1}(U)$ as
\begin{align}\label{local}
g = g(z) |\psi^*dt|^2.
\end{align}
Here $g(z) \not\equiv 0$ is a non-negative $C^0$ function on $V$ which is $C^2$ away from its zero set and shares the same support as that of $k$ on $V$, this support being independent of the choice of $h$.
\end{defi}

Note that we are grossly abusing the notation in using the same letter $g$ for the pseudometric $g$ on $C'$ and the function $g(z)$ which is only locally defined. Also, this local function $g(z)$ depends on the choice of the local coordinate $t$ on $C$, though no confusion should arise within normal contexts. 
The 2-tensor so defined is a semi-positive (real) sesquilinear 2-form on the tangent space of $C'$ at each point. \\[-1mm]

We denote the abstract support of $k$ by ${\mathcal U}_g:=\{z\,;k\ne 0\}\subset C'$. It is easily seen to be independent of the choice of the hermitian metric $h$ on $C$.\\[-3mm]

The main result of this section is the following 
\begin{lem}[Non-existence of ramified hyperbolic metric]\label{orbi-AS}
Let $\psi_{\mathbb{C}}:\, \mathbb{C}' \to \mathbb{C}$ be a holomorphic finite ramified covering map. And $g = g(z)|dz|^2$  is a ramified metric on $\mathbb{C}'$ with respect to $\psi_\mathbb{C}$. Then there cannot exist any positive number $\epsilon$ such that
\begin{align}\label{curv-orbi-AS}
\sqrt{-1}\, \partial_z\partial_{\bar{z}} \log g(z) \geq \epsilon \cdot g(z)
\end{align}
on $\mathbb{C}'$ in the sense of currents. 
\end{lem}
\begin{proof}
Suppose such a positive constant $\epsilon$ exists.
After replacing $\mathbb{C}'$ we can assume that $\psi_{\mathbb{C}}:\, \mathbb{C}' \to \mathbb{C}$ is a \emph{Galois covering}. Denote by $G:= \mathrm{Aut}(\mathbb{C}'/\mathbb{C})$ the Galois group. We use the same notation $g$ for the pull-back metric on the new cover $\mathbb{C}'$, which is apparently a ramified metric with respect to $\psi_{\mathbb{C}}$.\\
Now we want to construct an \emph{invariant metric} on $\mathbb{C}'$. Since $\psi_{\mathbb{C}}$ is finite, the Galois group $G$ is a finite group. Thus we can make the following averaging
\[
g_{\mathrm{inv}}:= \sum_{\xi \in G} \xi^*g.
\]
Since for two metrics $g_1$ and $g_2$, we have {\cite[Lemma~4]{Sch08}}
\[
dd^c\log (g_1 + g_2) \geq \frac{g_1}{g_1+g_2} dd^c\log g_1 + \frac{g_2}{g_1+g_2} dd^c\log g_2,
\]
we know that the curvature inequality (\ref{curv-orbi-AS}) still holds for the invariant metric $g_{\mathrm{inv}}$.\\
We claim that $g_{\mathrm{inv}}$ can be descended to a bounded metric on $\mathbb{C}$. It is easy to descend the metric near the point where $\psi_{\mathbb{C}}$ is \'etale. So one can localize the problem to the following situation: $\psi_{\mathbb{D}}:\,\mathbb{D} \to \mathbb{D},\,\, z \mapsto t := z^m$ for some $m \in \mathbb{Z}_{>1}$.\\
Note that we have
\[
g_{\mathrm{inv}}=g_{\mathrm{inv}}(z) |\psi^*_{\mathbb{D}}dt|^2.
\]  
From the definition of the ramified metrics we know that $g_{\mathrm{inv}}(z)$ is bounded near the origin.\\ 
In this case, the Galois group action of $\psi_{\mathbb{D}}$ can be written explicitly as following: $\mathrm{Aut}(C'/C) \cong \mathbb{Z}/m\mathbb{Z}$ and the group action is generated by $z \mapsto e^{\frac{2\pi i}{m}}\cdot z$. Then the Galois invariance of $g_{\mathrm{inv}}$ implies that the function $g_{\mathrm{inv}}(z)$ is Galois invariant, thus can be descended to some bounded function $g_{\mathbb{C}}(t)$.\\
Therefore, we conclude that $g_{\mathrm{inv}}$ can be descended to a bounded metric $g_{\mathbb{C}}(t)|dt|^2$ on $\mathbb{C}$. Then near those points of $\mathbb{C}$ where $\psi_{\mathbb{C}}$ is \'etale, the curvature inequality $\sqrt{-1}\,\partial_z\partial_{\bar{z}} \log g_{\mathrm{inv}}(z) \geq \epsilon \cdot g_{\mathrm{inv}}(z)$ descend to $\sqrt{-1}\,\partial_t\partial_{\bar{t}} \log g_{\mathbb{C}}(t) \geq \epsilon \cdot g_{\mathbb{C}}(t)$. Since $g_{\mathbb{C}}(t)$ is bounded, we can extend the curvature inequality in the sense of currents by the continuity on the whole $\mathbb{C}$. This contradicts to the usual Ahlfors-Schwarz lemma {\cite[Lemma~3.2]{Dem97}}.
\end{proof}
In fact, the finiteness condition of $\psi_{\mathbb{C}}$ in Lemma~\ref{orbi-AS} can be removed. To prove this, one need the following ramified version of the Ahlfors-Schwarz lemma:\\[-4mm]

\begin{lem}[Ramified Ahlfors-Schwarz Lemma]\label{ram-AS}
Let $C$ and $C'$ be two Riemannian surfaces and $\psi:\, C' \to C$ a ramified covering map. Let $g$ be a ramified pseudometric on $C'$ with respect to $\psi$ and defined by a conformal factor $k$. Suppose that $g$ has strictly negative curvature on ${\mathcal U}_g$, i.e., that the (Gaussian) curvature $\kappa_g$ of $g$ is bounded from above by a negative constant $-c<0$ there, or equivalently that on ${\mathcal U}_g$, with $g(z)$ as given locally by equation~\ref{local}, we have
\[
\frac {\Delta\,\log g(z)}{2g(z)}=:-\kappa_g \ge c,\ \ \Delta=4\frac{\partial}{\partial z}\frac{\partial}{\partial {\bar{z}}}.
\]
Suppose $C$ has a metric $h$ of constant curvature $-1$, i.e., $C$ is hyperbolic. Then we have
\[
g \le \frac{1}{c} \psi^*h.
\]
\end{lem}
Since in our application we only concern about the finite ramified $\psi_{\mathbb{C}}$, we put the proof of Lemma~\ref{ram-AS} in the appendix~\ref{app}.

\section{Compatibility of Deformation Higgs bundles}\label{compatibility}

In this paper we consider the moduli functor of polarized manifolds with a semi-ample canonical sheaf, defined over $\mathbb{C}$. More precisely, let $h \in \mathbb{Q}[T_1,T_2]$ with $h(\mathbb{Z} \times \mathbb{Z}) \subset \mathbb{Z}$ be a polynomial of degree $n$ in $T_1$. The moduli functor $\mathscr{M}_h$ is given by
\[
\mathscr{M}_h(\Spec\, \mathbb{C}) = \left\{ (\Gamma, \mathcal{H}) \bigg\arrowvert \,\, 
\begin{array}{c}
\textrm{ $\Gamma$ is a projective manifold, $\mathcal{H}$ is an ample invertible sheaf over $\Gamma$, $\omega_\Gamma$ is} \\ 
\textrm{semi-ample and $h(\alpha,\beta) = \chi(\mathcal{H}^{\alpha} \otimes \omega^{\beta}_{\Gamma})$ for all $\alpha,\beta \in \mathbb{N}$} 
\end{array}
\right\}\bigg\slash_{\simeq}.
\]
Viehweg proved in {\cite[Theorem~1.13]{vieh95}} the following result:
\begin{thm}
There exists a coarse quasi-projective moduli scheme $\mathcal{M}_h$ for the moduli functor $\mathscr{M}_h$.  
\end{thm}
This means, for any family $(f:\, V \to U, \mathcal{L}) \in \mathscr{M}_h(U)$, we have a classifying morphism $\varphi:\, U \to \mathcal{M}_h$.
\begin{rmk}
In the construction of the coarse moduli scheme $\mathcal{M}_h$, instead of using the usual polarization $\mathcal{H}^{\nu}$ of $\Gamma$, Viehweg uses two polarizations $\mathcal{H}^{\nu} \otimes \omega^e_{\Gamma}$ and $\mathcal{H}^{\nu+1} \otimes \omega^{e'}_{\Gamma}$ for suitable $e,e' \in \mathbb{N}$ (that's why we use the modified Hilbert polynomial $h$ with two variables). The purpose of this ``double polarization'' is to make the moduli functor $\mathscr{M}_h$ \emph{separated} (cf. {\cite[Definition~1.5]{vieh95}}). The coarse moduli scheme is constructed by taking the geometric quotient of certain Hilbert scheme.
The separatedness guarantees that the group action is proper. See {\cite[\S7]{vieh95}} for a discussion.    
\end{rmk}

Now let $(f:\,V \to U,\mathcal{L})$ be a smooth family of polarized manifolds with fixed Hilbert polynomial $h$, that is, an element in $\mathscr{M}_h(U)$. As we explained above, the family induces a classifying map $\varphi:\, U \to \mathcal{M}_h$. When the family has maximal variation, i.e. $\varphi$ is generically finite, one can show the Brody hyperbolicity of $U$ via the Viehweg-Zuo construction \cite{VZ-1}.\\ 
However, in the formulation of the relative isotriviality conjecture, the genral fiber of $\varphi$ could be positive dimensional subscheme, and what we try to prove is some sort of hyperbolicity of the image $\varphi(U)$. 
So the strategy is to apply the Viehweg-Zuo construction ``on the image $\varphi(U)$'', which will be made precise in the following.\\

Note that there is no universal family on $\mathcal{M}_h$. So one cannot apply Viehweg-Zuo's argument directly on the image. 
However, one can find some ramified covering of $\mathcal{M}_h$ which carries a universal family by the works of Koll\'ar, Viehweg et al. 
So we can find $U^{\#}$ finite over $\varphi(U)$, which carries a maximally varied family and fits into the diagram
\[
\xymatrix{
\mathbb{C}' \ar[r] \ar[d] & U' \ar[r] \ar[d] & U^{\#} \ar[d] \\
\mathbb{C} \ar[r]^{\gamma} & U \ar[r] & \varphi(U)
}
\]
where $\gamma:\, \mathbb{C} \to U$ is the entire curve we are concerned about and both $U' \to U$ and $\mathbb{C}' \to \mathbb{C}$ are the pull back of $U^{\#} \to \varphi(U)$. Then we want to use the Viehweg-Zuo construction on $U^{\#}$ as well as Lemma~\ref{orbi-AS} to conclude that the composition map $\mathbb{C}' \to U^{\#}$ is constant.\\

The problem is that $U'$ and $U^{\#}$ could be quite singular since $\mathcal{M}_h$ is highly singular in general. 
So next we shall recall a technical lemma of \cite{VZ}, which gives good birational models for those spaces. We first recall some terminology.
\begin{defi}
Given a family $V \to U$ we will call $\hat{V} \to \hat{U}$ a \emph{birational model} of it if there exists compatible birational morphisms $\delta:\, \hat{U} \to U$ and $\delta':\, \hat{V} \to V \times_{U} \hat{U}$.  
\end{defi}
\begin{defi}
Let $V \to U$ be a smooth projective morphism between quasi-projective manifolds.
\begin{itemize}
\item[a)] We call $f:\, X \to Y$ a \emph{partial compactification} of $V \to U$, if
  \begin{itemize}
  \item[i)] $X$ and $Y$ are quasi-projective manifolds, and $U \subset Y$.
  \item[ii)] $Y$ has a non-singular projective compactification $\bar{Y}$ such that $S:= Y \setminus U$ extends to a normal crossing divisor and such that $\mathrm{codim}(\bar{Y} \setminus Y) \geq 2$.
  \item[iii)] $f$ is a projective morphism and $f^{-1}(U) \to U$ coincides with $V \to U$.
  \item[iv)] $S:= Y \setminus U$ and $\Delta:=f^*S$ are normal crossing divisors.
\end{itemize}
\item[b)] We say that a partial compactification $f:\, X \to Y$ is \emph{good} if the condition iv) in a) is replaced by
  \begin{itemize}
  \item[iv)'] $f$ is flat, $S:= Y \setminus U$ is a smooth divisor and $\Delta:=f^*S$ is a relative normal crossing divisor, i.e. a normal crossing divisor whose components and all their intersections are smooth over components of $S$. 
  \end{itemize}
\item[c)] The good partial compactification $f:\, X \to Y$ is said to be \emph{semi-stable} if in b), iv)', the divisor $f^*S$ is reduced.
\item[d)] An arbitrary partial compactification of $V \to U$ is called semi-stable in codimension one, if it contains a semi-stable good partial compactification.
\end{itemize}
\end{defi}

\begin{defi}
A projective morphism $g:\, Z \to Y$ between quasi-projective varieties is called \emph{mild}, if
\begin{itemize}
\item[a)] $g$ is flat, Gorenstein with reduced fibers.
\item[b)] $Y$ is non-singular and $Z$ normal with at most rational singularities.
\item[c)] Given a dominant morphism $Y_1 \to Y$ where $Y_1$ has at most rational Gorenstein singularities, $Z \times_Y Y_1$ is normal with at most rational singularities.
\item[d)] Let $Y_0$ be an open subvariety of $Y$, with $g^{-1}(Y_0) \to Y_0$ smooth. Given a non-singular curve $C$ and a morphism $C \to Y$ whose image meets $Y_0$, the fiber product $Z \times_Y C$ is normal, Gorenstein with at most rational singularities. 
\end{itemize}
\end{defi}
The following is our key technical lemma
\begin{lem}[Lemma~2.8 in \cite{VZ}]\label{diagram}
Replacing the original family by some birational model, which we still denote as $V \to U$, one can find a good partial compactification $f:\, X \to Y$ of it, such that the partially compactified family $X \to Y$ fits into the following commutative diagram
\begin{align}
  \label{eq:2}
\xymatrix{
X \ar[d]^f  & X' \ar[d]^{f'} \ar[l]_{\psi'} & Z \ar[ld]^-g \ar[l]_{\sigma} \ar[r]^{\eta'} & Z^{\#} \ar[ld]^-{g^{\#}}  \\
Y & Y' \ar[l]_{\psi} \ar[r]^{\eta} & Y^{\#} &
}
\end{align}
with:
\begin{itemize}
\item[a)] $g^{\#}$ is a projective morphism, birational to a mild projective morphism ${g^{\#}}':\, {Z^{\#}} \to Y^{\#}$;
\item[b)] $g^{\#}$ is semi-stable in codimension one;
\item[c)] both $Y'$ and $Y^{\#}$ are smooth projective, $\eta$ is dominant and smooth in codimension one, $\eta'$ factors through a birational morphism $Z \to Z^{\#} \times_{Y^{\#}} Y'$, and $\psi$ is finite;
\item[d)] $X'$ is the normalization of $X \times_Y Y'$ and $\sigma$ is a blowing up with center in ${f'}^{-1}\psi^{-1}(S')$; in particular $f'$ and $g$ are projective;
\item[e)] let $U^{\#}$ be the largest subscheme of $Y^{\#}$ with
\[
V^{\#}:= {g^{\#}}^{-1}(U^{\#}) \to U^{\#}
\]
smooth; then $\psi^{-1}(U) \subset \eta^{-1}(U^{\#})$ and $U^{\#}$ is generically finite over $\mathcal{M}_h$.
\end{itemize}
\end{lem}
\begin{rmk}
Note that Lemma~2.8 in \cite{VZ} requires the original family to have the canonical polarization. In fact their argument also works for our situation without much modification. We outline the proof briefly here.   
\end{rmk}
\begin{proof}[Sketch of the proof]
By {\cite[Lemma~7.6]{vieh95}} we know that the moduli functor $\mathscr{M}_h$ has reduced finite automorphism ( this is also true for the moduli functor of canonically polarized manifolds, and this is the only place where \cite[Lemma~2.8]{VZ} uses the condition about canonical polarization ). That means we can use the theorem of Koll\'ar (cf. \cite{kol90} or {\cite[Theorem~9.25]{vieh95}}) to find a reduced normal scheme $Z$, a finite group $\Gamma$ acting on $Z$ and a family $(g:\, X \to Z, \mathcal{L}) \in \mathscr{M}_h(Z)$ such that the normalization of $(\mathcal{M}_h)_{\mathrm{red}}$ is isomorphic to the quotient $Z/\Gamma$, and the composition $Z \to Z/\Gamma \to \mathcal{M}_h$ is exactly the classifying map induced by the family $g$, which is descended from the universal family over the Hilbert scheme.\\
So we can take $U^{\#}:= Z \times_{\mathcal{M}_h} \varphi(U)$ and $V^{\#} \to U^{\#}$ is the induced family. Let $U'$ be the desingularization of the normalization of $U^{\#} \times_{\varphi(U)} U$. Then we get the following commutative diagram
\[
\xymatrix{
U' \ar[r]^{\eta} \ar[d]^{\psi} & U^{\#} \ar[d] \\
U \ar[r]^{\varphi} & \mathcal{M}_h.
}
\] 
Now we denote by $V' \to U'$ the pull back of the family $V \to U$ along $\psi$. Note that the classifying map of $V' \to U'$ factors through $U^{\#} \to \mathcal{M}_h$ by the commutative diagram above. Then the property of the universal family guarantees that $V' \cong V^{\#} \times_{U^{\#}} U'$.\\
Note that we can desingularize both $U'$ and $U^{\#}$ to keep the diagram commutative. After projective compactification of the families $f:\, X \to Y$, $g:\, X' \to Y'$ and $g^{\#}:\, Z^{\#} \to Y^{\#}$, we have the following diagram
\begin{align}\label{three-Y}
\xymatrix{
Y' \ar[r]^{\eta} \ar[d]^{\psi} & Y^{\#} \\
Y & 
}
\end{align}
of smooth projective varieties.
However, $\psi:\, Y' \to Y$ is no longer a finite morphism. It is projective and generically finite, that is, an \emph{alteration}. Here is a trick to make $\psi$ finite after replacing $Y$ by some birational model. First we may assume $Y' \to Y$ is a Galois alteration. Then consider the Stein factorization $Y' \to \tilde{Y}' \to Y$. Let $G$ be the Galois group of the alteration $Y' \to Y$. Next we define $\hat{Y}:= \Spec_{Y} (\psi_*\mathcal{O}_{Y'})^G$. Then $\hat{Y} \to Y$ is birational since they have the same function field. By functoriality we have
\[
\xymatrix@dr{
Y' \ar[r] \ar@{.>}[d]  & \tilde{Y}' \ar[d] \\
\hat{Y} \ar[r]  & Y
}
\]
where $Y' \to \hat{Y}$ is finite. In the following construction we shall change the birational model of $Y'$ and $Y^{\#}$ several times, and this trick can help us to transport the birational modification to $Y$ and preserve the finiteness of $\psi$.\\
Now we can apply {\cite[Lemma~2.3, b) and Lemma~2.6]{VZ}} to the family $g^{\#}:\,Z^{\#} \to Y^{\#}$ and obtain a mild model of it which is semi-stable in codimension one as required in a) and b) after replacing $Y^{\#}$ by some alteration. 
The new family which we still denote as $g^{\#}:\,Z^{\#} \to Y^{\#}$ has maximal variation since the classifying map from the smooth locus $U^{\#}$ to $\mathcal{M}_h$ is a composition of generically finite maps. Replacing $Y'$ by the desingularization of the fiber product will preserve the diagram (\ref{three-Y}) above ( of course the birational model of $Y$ is changed ).\\ 
Now we use the Stein factorization and replace $Y^{\#}$ by some finite cover. Then we can assume that $\eta$ has connected fibers. Next choose a blowing up $\hat{Y}^{\#} \to Y^{\#}$ such that the main component of $Y' \times_{Y^{\#}}\hat{Y}^{\#}$ is flat over $\hat{Y}^{\#}$. Hence desingularizing and replacing the notations again, we can assume that $\eta$ is flat and generically smooth since $Y'$ is nonsingular. Then we can apply {\cite[Lemma~2.3, b)]{VZ}} to $\eta:\,Y' \to Y^{\#}$ and obtain that $\eta$ is semi-stable in codimension one. As a consequence, $\eta$ is smooth in codimension one.\\
The check of the rest statements is routine, and details could be found in {\cite[proof of Lemma~2.8]{VZ}}.
\end{proof}

Thus we have three families of polarized manifolds $f:\,(X,\Delta) \to (Y,S)$, $g:\,(Z,\Delta') \to (Y',S')$ and $g^{\#}:\,(Z^{\#},\Delta^{\#}) \to (Y^{\#},S^{\#})$ in hand, whose geometry are closely related by Lemma~\ref{diagram}. \\[-3mm]

In the rest of this section, we apply this technical construction to show the \emph{compatibility of deformation Higgs bundles} associated to these families. This is crucial for our proof of the relative isotriviality conjecture. We first recall the definition of a deformation Higgs bundle.\\[-3mm]

As in \cite{VZ-1} and \cite{VZ}, we shall use the tautological short exact sequences
\begin{equation}\label{tautol-seq}
0 \to f^*\Omega^1_Y(\log S) \otimes \Omega^{p-1}_{X/Y}(\log \Delta) \to \mathfrak{gr}(\Omega^p_X(\log \Delta)) \to \Omega^p_{X/Y}(\log \Delta) \to 0
\end{equation}
where
\[
\mathfrak{gr}(\Omega^p_X(\log \Delta)) := \Omega^p_X(\log \Delta) / f^*\Omega^2_Y(\log S) \otimes \Omega^{p-2}_{X/Y}(\log \Delta).
\]
Note that the short exact sequence can be established only when $f:\,(X,\Delta) \to (Y,S)$ is log smooth. Denote by $\mathcal{L}=\Omega^n_{X/Y}(\log \Delta)$. We define
\[
F^{p,q}:= R^qf_*(\Omega^p_{X/Y}(\log \Delta) \otimes \mathcal{L}^{-1})/\mathrm{torsion}
\]
together with the edge morphisms
\[
\tau^{p,q}:\, F^{p,q} \to F^{p-1,q+1} \otimes \Omega^1_Y(\log S)
\]
induced by the exact sequence (\ref{tautol-seq}) tensored with $\mathcal{L}^{-1}$.  
\begin{rmk}
It is easy to see that $\tau^{n,0}|_U$ is nothing but the Kodaira-Spencer map of the family. So the Higgs maps $\tau^{p,q}$ can be regarded as the \emph{generalized Kodaira-Spencer maps}.  
\end{rmk}
So we get the Higgs sheaf $(F,\tau)$ defined on $Y$.
We define $(F',\tau')$ over $Y'$ and  $(F^{\#},\tau^{\#})$ over $Y^{\#}$ in the same manner.\\

By the smoothness of $\eta$ and the functoriality of the construction of $(F,\tau)$ under the pull back of families, we have the isomorphism
\[
\eta^*(F^{\#},\tau^{\#}) \cong (F',\tau')
\]
over a large open subset. More precisely, we have the following commutative diagram
\begin{align}\label{comm-1}
\xymatrix{
\eta^*(F^{\#})^{p,q} \ar[rr]^-{\eta^*(\tau^{\#})^{p,q}} \ar[d]^{\cong} & & \eta^*(F^{\#})^{p-1,q+1}  \otimes \eta^* \Omega^1_{Y^{\#}}(\log S^{\#}) \ar[d]^{\cong \otimes d\eta} \\
(F')^{p,q} \ar[rr]^-{(\tau')^{p,q}} & & (F')^{p-1,q+1} \otimes \Omega^1_{Y'}(\log S').
}
\end{align}
To show this one only needs to notice that over the large open subset (i.e. the complement has codimension $\geq$ 2 ) of $Y'$ where $\eta$ is smooth, we have $Z = Z^{\#} \times_{Y^{\#}} Y'$ and thus $\Omega^p_{Z/Y'}(\log \Delta') \cong \eta'^*\Omega^p_{Z^{\#}/Y^{\#}}(\log \Delta^{\#})$. Then commutative diagram (\ref{comm-1}) follows from the flat base change theorem.\\

Next we study the relation between the deformation Higgs bundles $(F,\tau)$ over $Y$ and $(F',\tau')$ over $Y'$. From {\cite[Lemma~4.1]{VZ}} we have
\begin{lem}\label{zeta}
Keep the notations as above. Let $Y'_1$ be the largest open subset in $Y'$ with $X \times_Y Y'_1$ normal. For all $p$ and $q$ we have morphisms
\[
\zeta_{p,q}:\, F^{p,q} \to (F')^{p,q}
\]
whose restriction to $Y'_1$ are isomorphisms, and the following diagram
\begin{align}\label{comm-2}
\xymatrix{
\psi^*F^{p,q} \ar[r]^-{\psi^*\tau^{p,q}} \ar[d]^{\zeta_{p,q}} &  \psi^*F^{p-1,q+1} \otimes \psi^*\Omega^1_Y(\log S) \ar[d]^{\zeta_{p-1,q+1} \otimes d\psi} \\
(F')^{p,q} \ar[r]^-{(\tau')^{p,q}} & (F')^{p-1,q+1} \otimes \Omega^1_{Y'}(\log S')
}
\end{align}
commutes.
\end{lem}
\begin{proof}[Sketch of the proof in \cite{VZ}]
Recall the construction of the family $g:\, Z \to Y'$ from the diagram
\begin{align}\label{family-g}
\xymatrix{
Z \ar[r]^-{\varphi} \ar[rd]^-{g} & \tilde{X} \ar[r]^-{\tilde{\varphi}} \ar[d]^-{\tilde{f}} & X \times_Y Y' \ar[r]^-{\pi_1} \ar[ld]^-{\pi_2} & X \ar[ld]^-f \\
  & Y' \ar[r]^{\psi} & Y &
}
\end{align}
where $\tilde{\varphi}$ is the normalization and $\varphi$ is the desingularization. Denote by $\psi':=\pi_1 \circ \tilde{\varphi} \circ \varphi :\, Z \to X$ the composition.
The discriminant loci of $\psi$ and $\psi'$ will be $\Delta(Y'/Y)$ and $\Delta(Z/X)$ respectively. After changing the birational models of $Y$ and $X$ we can assume that $S^{\dagger}:= S + \Delta(Y'/Y)$ and $\Delta^{\dagger}:= \Delta+\Delta(Z/X)$, as well as their preimages in $Y'$ and $Z$, are normal crossing divisors.\\
Now we come to the key observation appeared in {\cite[Lemma~3.8]{VZ}}
\[
\Omega^1_{X/Y}(\log \Delta) = \Omega^1_{X/Y}(\log \Delta^{\dagger})
\]
which is due to the fact that $\Delta(Z/X) = f^{-1}\Delta(Y'/Y)$.
Then the enlarged log cotangent sheaves behave well under the finite base changes
\[
\psi^*\Omega^1_Y(\log S^{\dagger}) = \Omega^1_{Y'}(\log \psi^*S^{\dagger}) \quad \quad \quad \quad \psi'^*\Omega^1_X(\log \Delta^{\dagger}) \subset \Omega^1_Z(\log \psi'^*\Delta^{\dagger})
\]  
by the generalized Hurwitz's formula,
where the second inclusion is an isomorphism on the largest open subset where $\psi'$ is finite. Thus outside the exceptional locus of $\varphi$ we have
\[
\psi'^*\Omega^1_{X/Y}(\log \Delta) =
\psi'^*\Omega^1_{X/Y}(\log \Delta^{\dagger}) \cong \Omega^1_{Z/Y'}(\log \psi'^*\Delta^{\dagger})
= \Omega^1_{Z/Y'}(\log \Delta').
\]
Now $\zeta_{p,q}$ is defined as the following composition
\[
\begin{array}{c}
\psi^*F^{p,q} = \psi^*R^qf_*(\Omega^p_{X/Y}(\log \Delta^{\dagger}) \otimes \mathcal{L}^{-1}) \xrightarrow{\cong} R^q{\pi_2}_*(\pi^*_1(\Omega^p_{X/Y}(\log \Delta^{\dagger}) \otimes \mathcal{L}^{-1})) \to \\
R^q\tilde{f}_*(\tilde{\varphi}^* (\pi^*_1(\Omega^p_{X/Y}(\log \Delta^{\dagger}) \otimes \mathcal{L}^{-1}))) \xrightarrow{\cong} R^qg_*(\varphi^*\tilde{\varphi}^*\pi^*_1(\Omega^p_{X/Y}(\log \Delta^{\dagger}) \otimes \mathcal{L}^{-1})) = (F')^{p,q} 
\end{array}
\]
where the first isomorphism is given by the flat base change theorem and the second morphism is an isomorphism on the largest open subset where the normalization $\tilde{\varphi}$ is an isomorphism, in particular over $Y'_1$. Then the commutativity of (\ref{comm-2}) follows from the compatibility of the tautological short exact sequences (\ref{tautol-seq}).
\end{proof}

\begin{cor}\label{Omega-factor}
Keep the notations as above. Then those maps
\[
\zeta_{p,q}:\, \psi^*F^{p,q} \to (F')^{p,q}
\]
are isomorphisms over the smooth locus $U'$.
\end{cor}
\begin{proof}
Note that the family $f:\, X \to Y$ is smooth over $U$. Thus the fiber product $X \times_Y Y'$ is smooth over $U'=\psi^{-1}(U)$. On the other hand, in the construction in Lemma~\ref{diagram} the covering space $Y'$ can be chosen to be smooth. Thus we know that the restriction of $X \times_Y Y'$ on $U'$ is smooth, thus normal. Then from Lemma~\ref{zeta} we know that
\[
\zeta_{p,q}:\, \psi^*F^{p,q} \to (F')^{p,q}
\] 
is an isomorphism over $U'$.
\end{proof}
\begin{rmk}
Lemma~\ref{zeta} and Corollary~\ref{Omega-factor} basically tells us that $(F',\tau') \cong \psi^*(F,\tau)$ over the smooth locus $U'$, which is crucial for the proof of Lemma~\ref{factor-through} in the next section. We will see that the defect along the boundary divisor $S'$ does not influence our argument.  
\end{rmk}

Here we also mention another application of Lemma~\ref{diagram}, which is about the Kodaira dimension of the determinant of direct image sheaves of the family.\\
We consider the determinant bundles of the direct image sheaves of some power of relative canonical bundles associated to $f:\, X \to Y$, the so-called Viehweg line bundle $A := \mathrm{det}(f_*\Omega^n_{X/Y}(\log \Delta)^{\otimes \nu})$. If the family $f$ has maximal variation, then we know that $A$ is big from {\cite[Corollary~3.6]{VZ}}. For a general family one could use Lemma~\ref{diagram} and obtain three line bundles
\begin{align}
  \label{eq:1}
 \left\{   \begin{array}{l}
  A := \mathrm{det}(f_*\Omega^n_{X/Y}(\log \Delta)^{\otimes \nu}) /Y \\
    A' := \mathrm{det}(g_*\Omega^n_{Z/Y'}(\log \Delta')^{\otimes \nu}) /Y' \\
    A^{\#} := \mathrm{det}(g^{\#}_*\Omega^n_{Z^{\#}/Y^{\#}}(\log \Delta^{\#})^{\otimes \nu}) /Y^{\#} 
  \end{array} \right.
\end{align}
The third named author and Xin Lu recently prove the following result, which can be regarded as an algebraic version of the relative isotriviality conjecture. The proof will be given in a forthcoming paper.
\begin{thm}
$\kappa(A) = \kappa(A') = \kappa(A^{\#}) =\mathrm{Var}(f)$. Moreover,
the restriction of $A$ on the fibers of the classifying map $U \to \mathcal{M}_h$ is trivial.
\end{thm}

%
%
%
%
%
%

\section{Viehweg-Zuo construction on $Y^{\#}$ and the Relative Isotriviality}
In this section we shall prove our first main result:
\begin{thm}[Relative Isotriviality]\label{main-thm}
Let $(f:\,V \to U,\mathcal{L})\in \mathscr{M}_h(U)$ be a smooth family of polarized manifolds with semi-ample canonical sheaves and fixed Hilbert polynomial $h$, which induces the classifying map $\varphi:\, U \to \mathcal{M}_h$. Let $\gamma:\, \mathbb{C} \to U$ be an entire curve. Then the image curve $\gamma(\mathbb{C})$ is contained in a fiber of $\varphi$.
In particular, if $U$ contains a Zariski dense entire curve, then the family $V \to U$ is isotrivial. 
\end{thm}

We argue by contradiction. Suppose that there exists an entire curve $\gamma:\, \mathbb{C} \to U$ which is not contained in the fiber of the classifying map. After replacing $U$ by the Zariski closure of the image curve we can assume that $\gamma(\mathbb{C})$ is Zariski dense in $U$.\\
From the construction in section~\ref{compatibility} we can compactify the family $V \to U$ as $f:\, X \to Y$, and the compactified family $f$ fits into the following commutative diagram
\[
\xymatrix{
X \ar[d]^f  & Z \ar[l] \ar[r] \ar[d]^g & Z^{\#} \ar[d]^{g^{\#}} \\
Y           & Y' \ar[l]_{\psi} \ar[r]^{\eta}  & Y^{\#}
}
\]
where $\psi$ is a finite morphism and the restriction $g^{\#}:\, V^{\#} \to U^{\#}$ on the smooth locus $U^{\#}$ of $Y^{\#}$ has a generically finite classifying map to $\mathcal{M}_h$ (cf, Lemma~\ref{diagram}). Note that in this process we may replace $U$ by some birational transformation of it. Since we assume $\gamma(\mathbb{C})$ is Zariski dense, one can always lift $\gamma$ to the new base space $U$.\\
So we define $\psi_{\mathbb{C}}:\, \mathbb{C}' \to \mathbb{C}$ to be the base change of the finite morphism $\psi:\, Y' \to Y$ along the holomorphic map $\gamma$, namely we have the following Cartesian diagram
\[
\xymatrix{
\mathbb{C}' \ar[r]^{\gamma'} \ar[d]^{\psi_{\mathbb{C}}} & Y' \ar[d]^{\psi} \\
\mathbb{C} \ar[r]^{\gamma} & Y.
}
\] 
Then $\psi_{\mathbb{C}}$ is a holomorphic finite ramified covering map with possibly infinite ramified points. By our assumption on $\gamma$ we know that the following composition map
\[
\mathbb{C}' \xrightarrow{\gamma'} U' \xrightarrow{\eta} U^{\#} \xrightarrow{\varphi^{\#}} \mathcal{M}_h
\]
is non-constant. Moreover, since $\varphi \circ \gamma(\mathbb{C})$ is Zariski dense in $\varphi(U)$, we know that the image $\eta \circ \gamma'(\mathbb{C}')$ is Zariski dense in $Y^{\#}$.\\

Now we can apply the Viehweg-Zuo construction on $Y^{\#}$ since the family $g^{\#}:\, V^{\#} \to U^{\#}$ has maximal variation. By {\cite[Lemma~6.3]{VZ-1}} (see also {\cite[Lemma~4.4]{VZ}}) we can find a Hodge bundle $(E^{\#},\theta^{\#})$ on $Y^{\#}$ with the following comparison maps
\[
\xymatrix{
(F^{\#})^{p,q} \ar[rr]^-{(\tau^{\#})^{p,q}} \ar[d]^{\rho^{p,q}_{\#}} & & (F^{\#})^{p-1,q+1} \otimes \Omega^1_{Y^{\#}}(\log S^{\#}) \ar[d]^{\rho^{p-1,q+1}_{\#} \otimes \iota} \\
(A^{\#})^{-1} \otimes (E^{\#})^{p,q} \ar[rr]^-{\mathrm{Id} \otimes (\theta^{\#})^{p,q}} & & (A^{\#})^{-1} \otimes (E^{\#})^{p-1,q+1} \otimes \Omega^1_{Y^{\#}}(\log (S^{\#} + T^{\#}))
}
\]
where $T^{\#}$ is the set of extra singularities (cf. {\cite[\S4]{VZ}} for the details of the construction of $(E^{\#},\theta^{\#})$).\\
Write $\gamma_{\#}:= \eta \circ \gamma'$. We have the following chain of maps
\begin{align}\label{chain-maps}
T_{\mathbb{C}'} \to (\gamma')^*T_{Y'}(-\log S') \to \gamma^*_{\#}T_{Y^{\#}}(-\log S^{\#}) \xrightarrow{\gamma^*_{\#}(\tau^{\#})^{n-1,1}} \gamma^*_{\#}(F^{\#})^{n-1,1} \xrightarrow{\rho^{n-1,1}_{\#}} \gamma^*_{\#} \left( (A^{\#})^{-1} \otimes (E^{\#})^{n-1,1} \right).
\end{align}

\begin{lem}\label{factor-through}
The composition map defined in (\ref{chain-maps}) actually factors through $\psi^*_{\mathbb{C}}T_{\mathbb{C}}$.  
\end{lem}
\begin{proof}
From commutative diagrams (\ref{comm-1}) and (\ref{comm-2}) we have
\[
\xymatrix{
\eta^*T_{Y^{\#}}(-\log S^{\#}) \ar[r] &  \eta^*(F^{\#})^{n-1,1} \ar[d]^{\cong} \\
T_{Y'}(-\log S') \ar[r] \ar[u] \ar[d] & (F')^{n-1,1} \\
\psi^*T_Y(-\log S) \ar[r] & \psi^*F^{n-1,1} \ar[u]^{\zeta_{n-1,1}}.
}
\]
The first square comes from the fact $\eta^*(F^{\#})^{n,0} = (F')^{n,0} = \mathcal{O}_{Y'}$ and the following commutative diagram
\[
\xymatrix{
\eta^*T_{Y^{\#}}(-\log S^{\#}) \ar[r] &  \eta^*(F^{\#})^{n-1,1} \otimes \eta^*\left( \Omega^1_{Y^{\#}}(\log S^{\#}) \otimes T_{Y^{\#}}(-\log S^{\#})  \right)  \ar[r] & \eta^*(F^{\#})^{n-1,1} \ar@{=}[d] \\
T_{Y'}(-\log S') \ar[r] \ar[u] \ar@{=}[d] & \eta^*(F^{\#})^{n-1,1} \otimes \eta^*\Omega^1_{Y^{\#}}(\log S^{\#}) \otimes T_{Y'}(-\log S') \ar[r] \ar[u] \ar[d] & \eta^*(F^{\#})^{n-1,1} \ar[d]^{\cong} \\
T_{Y'}(-\log S') \ar[r] & (F')^{n-1,1} \otimes \Omega^1_{Y'}(\log S') \otimes T_{Y'}(-\log S') \ar[r] & (F')^{n-1,1}
}
\] 
where the map $\eta^*(F^{\#})^{n-1,1} \to (F')^{n-1,1}$ is an isomorphism over a large open subset of $Y'$.\\
The second square is constructed in a similar manner, and it only commutes over $U'$ since $\zeta:\, \psi^*F \to F'$ is only an isomorphism over $U'$ ( cf. Corollary~\ref{Omega-factor} ). Note that the image curve $\gamma'(\mathbb{C}')$ is contained in $U'$. Thus the commutativity above is sufficient for our purpose.\\
Pulling it back to $\mathbb{C}'$ we obtain the following diagram
\[
\xymatrix{
  &  \gamma^*_{\#}T_{Y^{\#}}(-\log S^{\#}) \ar[r] & \gamma^*_{\#}(F^{\#})^{n-1,1} \ar[r] \ar[d]^{\cong} & \gamma^*_{\#}\left( (A^{\#})^{-1} \otimes (E^{\#})^{n-1,1} \right) \\
T_{\mathbb{C}'} \ar[d] \ar[r] & (\gamma')^*T_{Y'}(-\log S') \ar[u] \ar[r] \ar[d] & (\gamma')^*(F')^{n-1,1} \ar[ru]  &  \\
\psi^*_{\mathbb{C}}T_{\mathbb{C}} \ar[r] & (\gamma')^*\psi^*T_Y(-\log S) \ar[r] & (\gamma')^*\psi^*F^{n-1,1} \ar[u] &
}
\]
where all the squares commutes. Note that although the map $\eta^*(F^{\#})^{n-1,1} \to (F')^{n-1,1}$ is only an isomorphism over a large open subset of $Y'$, we can extend the composition map $(F')^{n-1,1} \to \eta^*\left( (A^{\#})^{-1} \otimes (E^{\#})^{n-1,1} \right)$ since both sides are reflexive sheaves over $Y'$.\\ 
The chain of maps (\ref{chain-maps}) is exactly the path on the top of this diagram. Clearly it factors through $\psi^*_{\mathbb{C}}T_{\mathbb{C}}$ from the bottom of this diagram.
\end{proof}
Now we iterate the Higgs maps in (\ref{chain-maps}) and get
\begin{align}
  \label{iter-chain-maps}
\begin{array}{c}
T^{\otimes k}_{\mathbb{C}'} \to (\gamma')^*\Sym^kT_{Y'}(-\log S') \to \gamma^*_{\#}\Sym^kT_{Y^{\#}}(-\log S^{\#}) \xrightarrow{\gamma^*_{\#}(\tau^{\#})^{n-k,k}} \\
\gamma^*_{\#}(F^{\#})^{n-k,k} \xrightarrow{\rho^{n-k,k}_{\#}} \gamma^*_{\#} \left( (A^{\#})^{-1} \otimes (E^{\#})^{n-k,k} \right).  
\end{array}
\end{align}
We say that the integer $m$ is the \emph{maximal length of iteration} if the composition map (\ref{iter-chain-maps}) for $k=m$ factors through
\begin{align}\label{iteration-m}
T^{\otimes m}_{\mathbb{C}'} \to \gamma^*_{\#} (A^{\#})^{-1} \otimes \mathrm{Ker}(\theta^{n-m,m}_{\gamma_{\#}})
\end{align}
where $\theta^{p,q}_{\gamma_{\#}} :\, \gamma^*_{\#} (E^{\#})^{p,q} \xrightarrow{\gamma^*_{\#}(\theta^{\#})^{p,q}} \gamma^*_{\#} (E^{\#})^{p-1,q+1} \otimes \gamma^*_{\#}\Omega^1_{Y^{\#}}(\log(S^{\#} + T^{\#})) \to \gamma^*_{\#} (E^{\#})^{p-1,q+1} \otimes \Omega^1_{\mathbb{C}'}$.

\begin{proof}[proof of Theorem~\ref{main-thm}]
Using the argument in {\cite[\S7]{VZ-1}} one can construct a suitable singular metric $g_{A^{\#}}$ on the ample line bundle $A^{\#}$ such that the pull back of the metric $g^{-1}_{A^{\#}} \otimes g_{hodge}$ via (\ref{iteration-m}) induces a complex Finsler (pseudo)metric $F$ on $Y^{\#}$ with strictly negative holomorphic sectional curvature along the (ramified) entire curve (see also {\cite[\S2.3]{DLSZ}} for a summary).\\
Using Lemma~\ref{factor-through} iteratively, we show that the map (\ref{iteration-m}) actually factors through
\[
\psi^*_{\mathbb{C}}T^{\otimes m}_{\mathbb{C}} \to \gamma^*_{\#} (A^{\#})^{-1} \otimes \mathrm{Ker}(\theta^{n-m,m}_{\gamma_{\#}}).
\]
By Definition~\ref{def_ram_metric} we know that the pull back of $F$ on $\mathbb{C}'$  is in fact a \emph{ramified metric} with respect to $\psi_{\mathbb{C}}$. 
The curvature property of $F$ gives us the inequality
\[
\sqrt{-1}\, \partial \bar{\partial}\, \log \Vert \gamma'_{\#}(z) \Vert^2_F \gtrsim \Vert \gamma'_{\#}(z) \Vert^2_F  dz \wedge d\bar{z}
\]
(cf. {\cite[Theorem~2.20]{DLSZ}}). Note that $\gamma'(\mathbb{C}') \cap S' = \emptyset$. Thus the defect of $\zeta_{p,q}$ along the boundary $S'$ in Corollary~\ref{Omega-factor} will not violate our curvature inequality.\\
This curvature inequality guarantees that one can find a positive cosntant $\epsilon$ such that
\[
\Delta \log f(z) \geq \epsilon \cdot f(z)
\]
holds on the complement of zero set of $\gamma^*_{\#}F$,
where $f(z)$ is the local function of $\gamma^*_{\#}F = f(z)|\psi^*_{\mathbb{C}}dt|^2$. But this contradicts to Lemma~\ref{orbi-AS}.
\end{proof}

\section{Borel-Chern Hyperbolicity}
In this section we shall prove Theorem~\ref{thm-C}. For any smooth algebraic curve $C$ and a holomorphic map $\gamma:\, C \to U$, we consider the following composition map
\[
C \xrightarrow{\gamma} U \xrightarrow{\varphi} \mathcal{M}_h.
\]
We first observe that
\begin{lem}\label{BC-to-alg}
Let $(f:\, V \to U,\mathcal{L})$ be a non-isotrivial family as in Theorem~\ref{thm-C}, and $\gamma:\, C \to U$ be a holomorphic curve as above. Then the graph $\Gamma_{\gamma}$ in $C \times U$ is not Zariski dense if we know the composed holomorphic $\varphi \circ \gamma$ is an algebraic morphism. 
\end{lem}
\begin{proof}
We denote by $\Gamma_{\varphi \circ \gamma}$ the graph of $\varphi \circ \gamma$ in $C \times \varphi(U)$. Then we have the following commutative diagram
\[
\xymatrix{
\Gamma_{\gamma} \ar@{^{(}->}[r] \ar[d] & C \times U \ar[d]^{\mathrm{id} \times \varphi} \\
\Gamma_{\varphi \circ \gamma} \ar@{^{(}->}[r] & C \times \varphi(U)
}
\]
where both vertical maps are surjective since $\varphi:\, U \to \varphi(U)$ is surjective. Now the algebraicity of $\varphi \circ \gamma$ implies that $\Gamma_{\varphi \circ \gamma}$ is not Zariski dense in $C \times \varphi(U)$. Here we use the fact that $\varphi(U)$ has positive dimension. Then by the surjectivity we know that $\Gamma_{\gamma}$ is not Zariski dense in $C \times U$.
\end{proof}
Therefore, we focus on the algebraicity of the composed map $\varphi \circ \gamma$.\\[.2cm]

If the moduli space $\mathcal{M}_h$ carries a universal family, then $\varphi \circ \gamma$ is an algebraic morphism by {\cite[Theorem~B]{DLSZ}} and GAGA. 
In general, there is no universal family over $\mathcal{M}_h$. So we shall use the construction in Lemma~\ref{diagram} again.\\[.2cm]

We argue by contradiction. Suppose $\varphi \circ \gamma$ is not algebraic. Then after replacing $U$ by the Zariski closure of $\gamma(C)$ we can assume that $\gamma(C)$ is Zariski dense in $U$. Note that the image of classifying map still has positive dimension.  Then we can desingularize the image $\varphi(U)$ and do the compatible birational transform of $U$. Since $\gamma(C)$ is Zariski dense, one can lift $\gamma$ to the new birational model of $U$, which will still be denoted as $U$ for simplicity. Thus we have the following diagram
\[
C \xrightarrow{\gamma} U \xrightarrow{\varphi} U^{\diamond}
\]
where $U^{\diamond} \to \mathcal{M}_h$ is generically finite.\\
Next we apply the construction in Lemma~\ref{diagram} to our situation. Then we obtain
\begin{align}\label{diagram-C}
\xymatrix{
C' \ar[r]^{\gamma'} \ar[d]^{\psi_C} & Y' \ar[r]^{\eta} \ar[d]^{\psi} & Y^{\#} \ar[d]^{\psi_{\diamond}} \\
C \ar[r]^{\gamma} & Y \ar[r]^{\varphi} & Y^{\diamond}
}
\end{align}
where $Y^{\diamond}$ is a smooth projective compactification of $U^{\diamond}$,
$C'$ is the fiber product $C \times_Y Y'$. Furthermore, from the proof of Lemma~\ref{diagram} we can assume $\psi$ and $\psi_{\diamond}$ to be Galois covers, and thus so is $\psi_C$.\\
The strategy is to show that the composition map $\varphi \circ \gamma$ extends to $\bar{C} \to Y^{\diamond}$, where $\bar{C}$is the projective completions of $C$. Since we only have a maximally varied family over the Galois cover $U^{\#}$, it is necessary to generalize the criterion in {\cite[Theorem~A]{DLSZ}} to the following ramified version:
\begin{prop}\label{ram-cri}
Let $(f:\, V \to U,\mathcal{L})$ be the family as in Theorem~\ref{thm-C} and $\gamma:\, C \to U$ be a holomorphic map from a smooth quasi-projective curve to the base $U$. By applying Lemma~\ref{diagram} one obtains the diagram (\ref{diagram-C}). Assume that there is a Finsler pseudometric $h$ on
$T_{Y^{\#}}(-\log S^{\#})$ such that the pull-back metric along $\gamma_{\#}:= \eta \circ \gamma'$ is a ramified metric on $C'$ with respect to $\psi_C$, the holomorphic sectional curvature of $h$ along $\gamma_{\#}$ is negatively bounded from above, and that the following inequality holds in the sense of currents
\begin{align}
  \label{curv-ineq-C}
dd^c \log \Vert \gamma'_{\#}(z) \Vert^2_h \geq \epsilon \cdot \gamma^*_{\#}\psi^*_{\diamond}\omega  
\end{align}
where $\omega$ is the Fubini-Study (1,1)-form on $Y^{\diamond}$ and $\epsilon$ is some positive constant. Then $\varphi \circ \gamma$ extends to a holomorphic map $\bar{C} \to Y^{\diamond}$.
\end{prop}
We now prove Theorem~\ref{thm-C} by using the criterion of Proposition~\ref{ram-cri}:
\begin{proof}[Proof of Theorem~\ref{thm-C}]
Let $\omega_{FS}$ be the Fubini-Study (1,1)-form on $Y^{\#}$.
Since $\psi_{\diamond}:\, Y^{\#} \to Y^{\diamond}$ is finite, one can find a positive constant $c$ such that
\[
\omega_{FS} \geq c \cdot \psi^*_{\diamond}\omega.
\] 
Note that $V^{\#} \to U^{\#}$ is a family with maximal variation. Thus we can apply {\cite[Theorem~2.19, Theorem~2.20]{DLSZ}} to the holomorphic map $\gamma_{\#}:\, C' \to Y^{\#}$ and find a Finsler (pseudo)metric $h$ constructed in the same manner as in (\ref{iteration-m}) which has holomorphic sectional curvature and satisfies the following curvature inequality
\[
dd^c \log \Vert \gamma'_{\#}(z) \Vert^2_h \gtrsim \gamma^*_{\#} \omega_{FS}.
\]
So we have
\[
dd^c \log \Vert \gamma'_{\#}(z) \Vert^2_h \geq \epsilon \cdot \gamma^*_{\#}\psi^*_{\diamond}\omega
\]
for some $\epsilon >0$. From Lemma~\ref{factor-through} we know that the pull-back metric $\gamma^*_{\#}h$ is in fact a ramified metric with respect to $\psi_C$. So all the conditions in Proposition~\ref{ram-cri} are satisfied, and thus $\varphi \circ \gamma$ extends. Now the Borel-Chern hyperbolicity follows from GAGA and Lemma~\ref{BC-to-alg}.
\end{proof}

The proof of Proposition~\ref{ram-cri} follows from the same line of reasoning as in {\cite[\S3]{DLSZ}}, plus the descend property of the ramified metrics.
\begin{proof}[Proof of Proposition~\ref{ram-cri}]
By our assumption we have the curvature inequality
\[
dd^c\log \Vert \partial_z \Vert^2_{\gamma^*_{\#}h} \geq \epsilon \cdot \gamma^*_{\#}\psi^*_{\diamond}\omega.
\]
Now since $\gamma^*_{\#}h$ is a ramified metric with respect to $\psi_C$ and $\psi_C$ is a finite Galois cover, one can use the method in the proof of Lemma~\ref{orbi-AS} to construct a Galois invariant metric $h_{\mathrm{inv}}$ with the same curvature inequality
\[
dd^c\log \Vert \partial_z \Vert^2_{h_{\mathrm{inv}}} \geq \epsilon' \cdot \gamma^*_{\#}\psi^*_{\diamond}\omega
\]
holding, and it can be descended to a bounded metric $h_C$ on $C$.
Thus near these points of $C$ where $\psi_C$ is \'etale, we have 
\[
dd^c \log h_C(t) \geq \epsilon' \cdot \gamma^*_{\diamond}\omega
\]
where $\gamma_{\diamond}:= \varphi \circ \gamma :\, C \to Y^{\diamond}$.\\
The assumption on the holomorphic sectional curvature implies that
\[
dd^c \log h_C(t) \geq c \cdot \sqrt{-1}\,h_C(t)dt \wedge d\bar{t}
\]
for some constant $c>0$. Now we define the semi-positive (1,1)-form
\[
\omega_{\gamma_{\diamond}} := \sqrt{-1}\,h_C(t)dt \wedge d\bar{t}
\]
on $C$. Then the two crucial curvature inequalities in {\cite[\S3]{DLSZ}} are satisfied for $\omega_{\gamma_{\diamond}}$. Thus the Nevanlinna theoretic argument in loc. cit. is applicable in our situation and gives us the extension of $\gamma_{\diamond}$.
\end{proof}

\appendix
\section{Ramified Ahlfors-Schwarz Lemma}\label{app}
In this appendix we shall give the proof of Lemma~\ref{ram-AS}.

\begin{proof}[Proof of Lemma~\ref{ram-AS}]
We define the ratio function $\displaystyle \phi:= \frac g {\psi^*h}$ on $C'$. It is non-negative, continuous and smooth on $\mathcal U$. We need to show that $\phi$ is bounded from above by $\displaystyle\frac1 c>0$. When $\phi$ is identically zero, i.e., when ${\mathcal U}= \emptyset$, there is nothing to prove. Hence we assume ${\mathcal U}\ne \emptyset$.\\[2mm]
Let $\mathbb{D} \to C$ be the uniformization by the Poincar\'e unit disc $\mathbb{D}= \{w \in \mathbb{C}\,;\,|w|<1\}$ and $\mathbb{D}':= \mathbb{D} \times_C C'$ be the fiber product. Then $h$ pulls back to the Poincar\'e metric on $\mathbb D$ and then to a ramified metric on ${\mathbb D}'$, all of constant curvature $-1$.
We denote by $\hat{g}$, $\hat{h}$ and $\hat{\phi}$ the pull back of $g$, $\psi^*h$ and $\phi$ to the ``ramified disc'' $\mathbb{D}'$. 
Define $\mathbb{D}_{\varepsilon}:= \{z \in \mathbb{C}\,;\, |z|<1-\varepsilon\}$
for $0 < \varepsilon \ll 1$
and consider the following Cartesian diagram
\[
\xymatrix{
\mathbb{D}'_{\varepsilon} \ar@{^{(}->}[r] \ar[d]^{\psi_{\varepsilon}} & \mathbb{D}' \ar[d]^{\psi_{\mathbb{D}}} \\
\mathbb{D}_{\varepsilon} \ar@{^{(}->}[r] & \mathbb{D}.
}
\]
Denote by $\hat{h}_{\varepsilon}$ the pull back of the Poincar\'e metric on $\mathbb{D}_{\varepsilon}$ via $\psi_{\varepsilon}$. Then the new ratio function 
\[
\hat{\phi}_{\varepsilon}:= \frac{\hat{g}|_{\mathbb{D}'_{\varepsilon}}}{\hat{h}_{\varepsilon}}
\]
decays to zero as $z$ approaches $\partial \mathbb{D}'_{\varepsilon}=\{z\,;|z|=1\}$ by the completeness of the Poincar\'e metric. Thus $\hat{\phi}_{\varepsilon}$ can be regarded as a well-defined function on $\mathbb{D}'$ by extending it by zero and we have $\lim_{\varepsilon \to 0} \hat{\phi}_{\varepsilon} = \hat{\phi}$. 
Now $\hat{\phi}_{\varepsilon}$ is positive on the nonempty open subset $\mathbb{D}'\cap {\mathcal U}$ and has compact support in $\mathbb{D}'_{\varepsilon}$. Hence $\hat{\phi}_{\varepsilon}$ (as well as its logarithm defined on this open subset) \emph{attains its maximum at some point} $w_0$ there by its continuity. The Hessian matrix of $\log \hat{\phi}_{\varepsilon}$ at $w_0$ is thus negative semi-definite. Hence, so is its trace $\Delta\,\log \hat{\phi}(w_0)$ and we get
\[
0 \geq \Delta\,\log \hat{\phi}_{\varepsilon}(w_0) = \Delta\,\log \hat{g}(w_0) - \Delta\,\log \hat{h}_{\varepsilon}(w_0)
\geq 2c \cdot \hat{g}(w_0) - 2\hat{h}_{\varepsilon}(w_0).
\]
We have used the functoriality of the curvature operator with respect to the pull back operation on metrics and that, by defintion, the Poincar\'e metric has curvature $-1$ on any disc. This implies
\[
\hat{\phi}_{\varepsilon}(w) \leq \mathrm{sup}\,\hat{\phi}_{\varepsilon} = \hat{\phi}_{\varepsilon}(w_0) \leq \frac{1}{c}. 
\]
Since $\lim_{\varepsilon \to 0}\hat{\phi}_{\varepsilon} = \hat{\phi}$ on $\mathbb{D}'$, it follows that $\hat{\phi}$ is bounded from above by $\displaystyle\frac 1 {c}$ on $\mathbb{D}'$. Therefore, so is $\phi$ on $C'$.
\end{proof}

The following corollary of the ramified Ahlfors-Schwarz lemma is completely standard.
\begin{cor}
Let $\psi_{\mathbb{C}}:\, \mathbb{C}' \to \mathbb{C}$ be a holomorphic ramified covering map. Let $g = g(z)|\psi^*dt|^2$  be a ramified metric on $\mathbb{C}'$ with respect to $\psi_\mathbb{C}$ with nonempty ${\mathcal U}_g$. Then $g$ does not have strictly negative curvature on ${\mathcal U}_g$.
\end{cor}
\begin{proof}
Suppose that such a ramified pseudometric $g$ exists and ${\mathcal U}_g\ne \emptyset$. 
After rescaling of $g$ we can assume that $\kappa_g\le -1$. For a positive real number $R$ we denote the disc of radius $R$ by $\mathbb{D}_R:=\{t \in \mathbb{C}\,\,|\,\, |t| < R\}$ and its $\psi$-preimage by $\mathbb{D}'_R:=\psi^{-1}_{\mathbb{C}}(\mathbb{D}_R)$. Then the Poincar\'e metric $h$ on $\mathbb{D}_R$ is
\[
h = \frac{R^{-2}|dt|^2}{(1-|t|^2/R^2)^2}.
\]
and has curvature $\kappa_h=-1$.
Denote by $\psi$ the restriction of $\psi_{\mathbb{C}}$ on $\mathbb{D}'_R$. Then we can apply the ramified Ahlfors-Schwarz Lemma~\ref{ram-AS} to $\psi:\, \mathbb{D}'_R \to \mathbb{D}_R$ to obtain the inequality
\[
g \leq \frac{R^{-2}|\psi^*dt|^2}{(1-|\psi^*t|^2/R^2)^2}
\]
on $\mathbb{D}'_R$. Letting $R$ approach $\infty$ shows that $g \equiv 0$, a contradiction. 
\end{proof}


\begin{thebibliography}{GGLR19} 
	\providecommand{\url}[1]{\texttt{#1}}
	\providecommand{\urlprefix}{URL }
	

	
        \bibitem[Cam11]{Cam11}
Fr\'{e}d\'{e}ric Campana.
        \newblock Orbifoldes g\'{e}om\'{e}triques sp\'{e}ciales et classification
  bim\'{e}romorphe des vari\'{e}t\'{e}s k\"{a}hl\'{e}riennes compactes.
        \newblock {\em J. Inst. Math. Jussieu}, 10(4):809--934, 2011.
        \newblock \urlprefix\url{https://doi.org/10.1017/S1474748010000101}.


	
	\bibitem[Dem97a]{Dem97}
	Jean-Pierre Demailly.
	\newblock \enquote{Algebraic criteria for {K}obayashi hyperbolic projective
		varieties and jet differentials.}
	\newblock In \enquote{Algebraic geometry---{S}anta {C}ruz 1995,} vol.~62 of
	\emph{Proc. Sympos. Pure Math.}, 285--360,  (Amer. Math. Soc., Providence,
	RI1997).
	\newblock \urlprefix\url{http://dx.doi.org/10.1090/pspum/062.2/1492539}.
	

	
	\bibitem[Den19b]{Den19}
	---{}---{}---.
	\newblock \enquote{{Hyperbolicity of coarse moduli spaces and isotriviality for
			certain families}.}
	\newblock \emph{arXiv e-prints} (2019) arXiv:1908.08372.
	


\bibitem[DLSZ]{DLSZ}
	Ya Deng, Steven Lu, Ruiran Sun and Kang Zuo.
	\newblock \enquote{Picard theorems for moduli spaces of polarized varieties.}
	(2020).
	\newblock ArXiv:1911.02973v3.
	

	
	\bibitem[Kol90]{kol90}
        J\'{a}nos Koll\'{a}r.  
	\newblock \enquote{Projectivity of complete moduli.}
	\newblock \emph{J. Differential Geom.}32 (1990), no. 1, 235--268. 
	\newblock \urlprefix\url{http://projecteuclid.org/euclid.jdg/1214445046}.


	\bibitem[LY90]{LY}
        Lu, Steven Shin-Yi and Yau, S.-T.  
	\newblock \enquote{Holomorphic curves in surfaces of general type.}
	\newblock \emph{Proc. Nat. Acad. Sci. U.S.A.} 87 (1990), no. 1, 80--82.
	\newblock \urlprefix\url{https://doi.org/10.1073/pnas.87.1.80}.


	\bibitem[LW12]{LW}
	Lu, Steven S. Y. and Winkelmann, J\"{o}rg.
	\newblock \enquote{Quasiprojective varieties admitting {Z}ariski dense entire
              holomorphic curves.}
	\newblock \emph{Forum Math.} 24 (2012), no. 2, 399--418.
	\newblock \urlprefix\url{https://doi.org/10.1515/form.2011.069}.

\bibitem[Sch08]{Sch08}
        Georg Schumacher.
	\newblock \enquote{Positivity of relative canonical bundles for families of canonically polarized manifolds.}
	(2008).
	\newblock ArXiv:0808.3259.


	\bibitem[Vie95]{vieh95}
	Viehweg, Eckart.
	\newblock \emph{Quasi-projective moduli for polarized manifolds}, vol. 30 of \emph{Ergebnisse der Mathematik und ihrer Grenzgebiete (3) [Results
              in Mathematics and Related Areas (3)]},  (Springer-Verlag, Berlin1995).
	\newblock \urlprefix\url{https://doi.org/10.1007/978-3-642-79745-3}.

	

	
	\bibitem[VZ02]{VZ}
	Eckart Viehweg and Kang Zuo.
	\newblock \enquote{Base spaces of non-isotrivial families of smooth minimal
		models.}
	\newblock In \enquote{Complex geometry ({G}\"ottingen, 2000),} 279--328,
	(Springer, Berlin2002).
	
	\bibitem[VZ03]{VZ-1}
	---{}---{}---.
	\newblock \enquote{On the {B}rody hyperbolicity of moduli spaces for
		canonically polarized manifolds.}
	\newblock \emph{Duke Math. J.} (2003) vol. 118~(1): 103--150.
	\newblock \urlprefix\url{http://dx.doi.org/10.1215/S0012-7094-03-11815-3}.
	

	
\end{thebibliography}
\end{document}